\documentclass[11pt]{article}
\usepackage{amssymb}
\usepackage{amsmath}
\usepackage{amsthm}
\usepackage{latexsym}
\usepackage{hyperref}
\usepackage{mathdots}
\usepackage[all]{xy}
\usepackage{stmaryrd}
\usepackage[normalem]{ulem}
\usepackage[capitalise]{cleveref}
\usepackage{tikz}
\usetikzlibrary{patterns}

\theoremstyle{definition}
\newtheorem{theo}{Theorem}[section]
\newtheorem{defi}[theo]{Definition}
\newtheorem{prop}[theo]{Proposition}
\newtheorem{defiprop}[theo]{Definition/Proposition}
\newtheorem{cor}[theo]{Corollary}

\newtheorem{exa}[theo]{Example}
\newtheorem{rem}[theo]{Remark}
\numberwithin{equation}{section}

\newcommand{\N}{{\mathbb N}}
\newcommand{\F}{{\mathbb F}}
\newcommand{\A}{{\mathbb A}}

\newcommand{\Z}{{\mathbb Z}}
\newcommand{\Q}{{\mathbb Q}}

\newcommand{\cC}{{\mathcal C}}
\newcommand{\cD}{{\mathcal D}}

\newcommand{\cI}{{\mathcal I}}

\newcommand{\cM}{{\mathcal M}}

\newcommand{\cB}{{\mathcal B}}

\newcommand{\cU}{{\mathcal U}}
\newcommand{\cV}{{\mathcal V}}

\newcommand{\cZ}{{\mathcal Z}}

\newcommand{\T}{\mbox{$^{\sf T}$}}
\newcommand{\subspace}[1]{\mbox{$\langle{#1}\rangle$}}
\newcommand{\inner}[2]{\mbox{$\langle{#1}\!\mid\!{#2}\rangle$}}
\newcommand{\Hom}{\mbox{\rm Hom}}
\newcommand{\Fnm}{\mbox{$\F^{n\times m}$}}

\newcommand{\rk}{{\rm rk}}
\newcommand{\drk}{\mbox{${\rm d}_{\rm rk}$}}
\newcommand{\dd}{\mbox{${\rm d}$}}
\newcommand{\tr}{{\rm tr}}
\newcommand{\cc}{{\rm c}}
\newcommand{\GL}{{\rm GL}}
\newcommand{\colsp}{\mbox{\rm colsp}}

\newcommand{\cIm}{\mbox{$\cI_{\text{max}}$}}
\newcommand{\qPM}{$q$-PM}

\newcounter{alp}
\newcounter{ara}
\newcounter{rom}
\newenvironment{romanlist}{\begin{list}{(\roman{rom})\hfill}{\usecounter{rom}
     \topsep.7ex \labelwidth.6cm \leftmargin.6cm \labelsep0cm
     \rightmargin0cm \parsep0ex \itemsep.5ex
     \partopsep0ex}}{\end{list}}
\newenvironment{alphalist}{\begin{list}{(\alph{alp})\hfill}{\usecounter{alp}
     \topsep-1.4ex \labelwidth.7cm \leftmargin.7cm \labelsep0cm
     \rightmargin0cm \parsep0ex \itemsep.5ex
     \partopsep-1.4ex}}{\end{list}}

\newenvironment{mylist}{\begin{list}{(\arabic{ara})\hfill}{\usecounter{ara}
     \topsep-1ex \labelwidth.73cm \leftmargin.73cm \labelsep0cm
     \rightmargin0cm \parsep0ex \itemsep.5ex
     \partopsep-1ex}}{\end{list}}
\newenvironment{mylist2}{\begin{list}{(\arabic{ara})\hfill}{\usecounter{ara}
     \topsep-1ex \labelwidth.88cm \leftmargin.88cm \labelsep0cm
     \rightmargin0cm \parsep0ex \itemsep.5ex
     \partopsep-1ex}}{\end{list}}
\newenvironment{mylist3}{\begin{list}{(\arabic{ara})\hfill}{\usecounter{ara}
     \topsep-1ex \labelwidth1cm \leftmargin1cm \labelsep0cm
     \rightmargin0cm \parsep0ex \itemsep.5ex
     \partopsep-1ex}}{\end{list}}
\makeatletter
\let\@fnsymbol\@arabic
\makeatother
\begin{document}
		
\title{Independent Spaces of $q$-Polymatroids}
\author{Heide Gluesing-Luerssen\thanks{Department of Mathematics, University of Kentucky, Lexington KY 40506-0027, USA; heide.gl@uky.edu. HGL was partially supported by the grant \#422479 from the Simons Foundation.}\quad and Benjamin Jany\thanks{Department of Mathematics, University of Kentucky, Lexington KY 40506-0027, USA; benjamin.jany@uky.edu.}}

\date{May 3, 2021}
\maketitle
	
\begin{abstract}
\noindent 
This paper is devoted to the study of independent spaces of $q$-polymatroids.
With the aid of an auxiliary $q$-matroid it is shown that the collection of independent spaces satisfies the same properties as 
for $q$-matroids.
However, in contrast to $q$-matroids, the rank value of an independent space does not agree with its dimension.
Nonetheless, the rank values of the independent spaces fully determine the $q$-polymatroid, and this fact can 
be exploited to derive a cryptomorphism of $q$-polymatroids.
Finally, the notions of minimal spanning spaces, maximally strongly independent spaces, and bases will be elaborated on.
\end{abstract}

\section{Introduction}\label{S-Intro}
Thanks to their relation to rank-metric codes, $q$-matroids and $q$-polymatroids have recently garnered a lot of attention, \cite{BCIJ21,BCIJS20,BCJ21,GhJo20,GLJ21,GJLR19,JuPe18,Shi19}. 
Indeed, $\F_{q^m}$-linear rank-metric codes in $\F_{q^m}^n$ give rise to $q$-matroids, whereas 
$\F_q$-linear rank-metric codes induce $q$-polymatroids.
This leads to an abundance of examples of $q$-(poly)matroids.
In either case, the $q$-(poly)matroid induced by a rank-metric code arises via a rank function which captures the dimension of 
certain characteristic subspaces of the code in question.
As a consequence, the $q$-(poly)matroid reflects many of the algebraic and combinatorial properties of the code, such as
the generalized weights~\cite{GLJ21,GJLR19} and the rank-weight enumerator~\cite{BCIJ21,Shi19}.

For $q$-matroids a variety of cryptomorphic definitions are known~\cite{JuPe18,BCJ21}.
They are based on independent spaces, bases, circuits, spanning spaces, flats and many more; see the comprehensive 
account in \cite{BCJ21}.
For $q$-polymatroids, most of these notions have yet to be defined.
As to our knowledge the only existing notion are flats, which have been introduced in \cite{GLJ21}.

In this paper we introduce the notion of independent spaces and bases for $q$-polymatroids.
As it turns out, the `standard notion' of independence, namely the equality of rank value and dimension, is too restrictive for
$q$-polymatroids. 
For this reason we introduce a more general notion of independence, which is inspired by the analogue 
for classical polymatroids in  \cite[Sec.~11]{Ox11}. 
In order to derive properties of the independent spaces, we introduce an auxiliary $q$-matroid on the same ground space
(akin to a construction in  \cite{Ox11}). 
Since the independent spaces of the auxiliary $q$-matroid coincide with those of the $q$-polymatroid, the latter inherits all 
properties known for independent spaces of $q$-matroids --- as long as these properties do not involve the rank function.

The independent spaces naturally give rise to a notion of basis, namely the maximal-dimensional independent 
subspaces.
Despite the lack of rigidity of the rank function in a $q$-polymatroid, it turns out that all bases of a subspace have the same rank value and this value agrees with the rank value of the subspace. 
In other words, the rank function restricted to the independent spaces fully determines the $q$-polymatroid.
This result allows us to provide a cryptomorphism for $q$-polymatroids based on independent spaces: we
characterize the collections of spaces endowed with a rank function on these spaces that give rise to a $q$-polymatroid with exactly these spaces as independent spaces and whose rank function restricts to the given one.
Examples show that no such cryptomorphism is possible using only bases, dependent spaces, or  circuits. 

We finally turn to spanning spaces. These are the spaces that share the same rank value as the ground space. 
It turns out that in a $q$-polymatroid every minimal spanning space is contained in a basis, but is, in general, not a basis itself.
Thus the notions `minimal spanning' and `maximally independent'  do not agree.
On the plus side, `minimal spanning' is the dual notion to `maximally strongly independent'.
This simple fact may be regarded as the generalization of the duality result for bases in $q$-matroids.
The latter states that in a $q$-matroid a space is a basis if and only if its orthogonal is a basis of the dual $q$-matroid.
It turns out that this equivalence is never true in a proper $q$-polymatroid and therefore 
characterizes $q$-matroids.

\textbf{Notation:}
We fix a finite field $\F=\F_q$ with~$q$ elements and a finite-dimensional $\F$-vector space~$E$.
We write $V\leq E$ if~$V$ is a subspace of~$E$ and denote by $\cV(E)$ the collection of all subspaces of~$E$.
The standard basis vectors in $\F^n$ are denoted by $e_1,\ldots,e_n$.
We write $[n]$ for the set $\{1,\ldots,n\}$.

\section{Basic Notions of $q$-Polymatroids}\label{S-Prelims}
In this section we define $q$-polymatroids and present some basic properties. We
 also introduce the main class of examples, namely $q$-polymatroids induced by rank-metric codes. 
The section is based on the material in \cite{GLJ21}.

\begin{defi}\label{D-PMatroid}
Set $\cV=\cV(E)$.
A \emph{$q$-rank function} on~$E$ is a map  $\rho: \cV\longrightarrow\Q_{\geq0}$  satisfying:
\begin{mylist2}
\item[(R1)\hfill] Dimension-Boundedness: $0\leq\rho(V)\leq \dim V$  for all $V\in\cV$;
\item[(R2)\hfill] Monotonicity: $V\leq W\Longrightarrow \rho(V)\leq \rho(W)$  for all $V,W\in\cV$;
\item[(R3)\hfill] Submodularity: $\rho(V+W)+\rho(V\cap W)\leq \rho(V)+\rho(W)$ for all $V,W\in\cV$.
\end{mylist2}
A \emph{$q$-polymatroid (\qPM) on~$E$} is a pair $\cM=(E,\rho)$, where $\rho: \cV\longrightarrow\Q_{\geq0}$ is a 
$q$-rank function.
The value $\rho(E)$ is called the \emph{rank} of the \qPM.
A number $\mu\in\Q_{>0}$ is called a \emph{denominator} of~$\rho$ (and $\cM$) if 
 $\mu\rho(V)\in\N_0$ for all $V\in\cV$.
In that case we call the map $\tau_{\mu}:=\mu\rho$ the \emph{induced integer $\rho$-function} w.r.t.~$\mu$.
The smallest denominator is called the \emph{principal denominator}.
A \qPM{} with principal denominator~$1$ (i.e., $\rho(V)\in\N_0$ for all~$V$) is called a \emph{$q$-matroid}.
If~$(E,\rho)$ is the \emph{trivial} \qPM, i.e.,~$\rho$ is the zero map, we declare~$1$ to be its principal denominator.
\end{defi}

We will often make use of the induced integer $\rho$-function $\tau_{\mu}$ for a given denominator~$\mu$.
Clearly~$\tau_\mu$ is also monotonic and submodular, and instead of (R1) it satisfies $0\leq\tau_\mu(V)\leq\mu\dim V$ for all $V\in\cV$.

The above definition appears in various forms in the literature. 
In \cite[Def.~4.1]{GJLR19} of Gorla et al.\ the same definition occurs with the only difference that the rank function may 
assume arbitrary real numbers. 
Next, a $q$-matroid in the sense of Jurrius/ Pellikaan~\cite{JuPe18} is exactly a $q$-matroid as defined above. 
Finally, for any $r\in\N$ a $(q,r)$-polymatroid as in \cite[Def.~2]{Shi19} by Shiromoto and \cite[Def.~1]{GhJo20} by Ghorpade/Johnson 
and \cite[Def.~1]{BCIJ21} by Byrne et al.\ can be turned into a \qPM{} with denominator~$r$ by dividing the rank function by~$r$. 
Conversely, given a \qPM{} $(E,\rho)$ with denominator~$\mu$, then $(E,\mu\rho)$ is a $(q,\lceil\mu\rceil)$-polymatroid in the sense of these papers.

It is possible for a \qPM{} that no nonzero subspace attains the upper bound in~(R1). 
In \cite{GLJ21} we call such \qPM{}s non-exact and discuss some basic facts about non-exact \qPM{}s.

\begin{rem}[\mbox{\cite[Rem.~2.3(3)]{GLJ21}}]\label{R-PrDenominator}
Let $(E,\rho)$ be a non-trivial \qPM{} with principal denominator~$\mu$.
Then the set of all denominators of $(E,\rho)$ is $\mu\N$.
\end{rem}

The following basic properties are presented for $q$-matroids in \cite[Prop.~6 and~7]{JuPe18}.
They hold true for \qPM{}s as well, and the proofs are identical to those in \cite{JuPe18}.

\begin{prop}\label{P-RankVx}
Let $(E,\rho)$ be a \qPM.
\begin{alphalist}
\item Let $V,W\in\cV(E)$.  Suppose $\rho(V+\subspace{x})=\rho(V)$ for all $x\in W$.
         Then $\rho(V+W)=\rho(V)$.
\item Let $V\in\cV(E)$ and $X,Y\in\cV(E)$ be $1$-dimensional spaces such that $\rho(V)=\rho(V+X)=\rho(V+Y)$.
        Then $\rho(V+X+Y)=\rho(V)$.
\end{alphalist}
\end{prop}

In order to discuss duality as well as some details on independent spaces we need the following notions of equivalence.
They extend \cite[Def.~4.4]{GJLR19}.
\qPM{}s are scaling-equivalent if they differ only by a scalar factor of the rank functions and an isomorphism between the ground spaces, and they are equivalent if the scalar factor is~$1$.

\begin{defi}\label{D-EquivMatroid}
Let $E_i,\,i=1,2,$ be $\F$-vector spaces of the same finite dimension and let
$\cM_i=(E_i,\rho_i)$ be \qPM{}s.
\begin{alphalist}
\item $\cM_1$ and $\cM_2$ are called \emph{scaling-equivalent} if there exists an $\F$-isomorphism 
         $\alpha\in\Hom_{\F}(E_1,E_2)$  
         and $a\in\Q_{>0}$ such that $\rho_2(\alpha(V))=a\rho_1(V)$ for all $V\in\cV(E_1)$.
\item  We call $\cM_1$ and $\cM_2$ \emph{equivalent}, denoted by $\cM_1\approx\cM_2$, if there exists an $\F$-isomorphism 
         $\alpha\in\Hom_{\F}(E_1,E_2)$ such that
         $\rho_2(\alpha(V))=\rho_1(V)$ for all $V\in\cV(E_1)$.
\end{alphalist}
\end{defi}

Clearly, both types of equivalences are indeed equivalence relations.

\begin{theo}[\mbox{\cite[Thm.~2.8]{GLJ21} and \cite[4.5--4.7]{GJLR19}, and \cite[Thm.~42]{JuPe18} for $q$-matroids}]\label{T-DualqPM}
Let $\inner{\cdot}{\cdot}$ be a non-degenerate symmetric bilinear form on~$E$. 
For $V\in\cV(E)$ define $\mbox{$V^\perp=\{w\in E\mid \inner{v}{w}=0$}\text{ for all }v\in V\}$.
Let $\cM=(E,\rho)$ be a \qPM{} and set
\[
    \rho^*(V)=\dim V+\rho(V^\perp)-\rho(E).
\]
Then $\rho^*$ is a $q$-rank function on~$E$ and $\cM^*=(E,\rho^*)$ is a \qPM. It is called the \emph{dual} of~$\cM$ with respect to the
form $\inner{\cdot}{\cdot}$. 
Furthermore, the bidual $\cM^{**}=(\cM^*)^*$ satisfies $\cM^{**}=\cM$, and $\cM$ and $\cM^*$ have the same set of denominators.
Finally, the equivalence class of~$\cM^*$ does not depend on the choice of the bilinear form.
More precisely, if $\langle\!\inner{\cdot}{\cdot}\!\rangle$ is another non-degenerate symmetric bilinear form on~$E$ and 
$\cM^{\hat{*}}=(E,\rho^{\hat{*}})$ is the resulting dual \qPM{}, then $\cM^{\hat{*}}\approx\cM^*$.
\end{theo}

The next result has been proven in~\cite{GJLR19} for \qPM{}s on $\F^n$, endowed with the standard dot product.
Thanks to the  invariance of the dual, it generalizes without the need to specify bilinear forms.

\begin{prop}[\mbox{\cite[Prop.~4.7]{GJLR19}}]\label{P-EquivDual}
Let $\cM=(E,\rho)$ and $\hat{\cM}=(\hat{E},\hat{\rho})$ be \qPM{}s.
Then $\cM\approx\hat{\cM}$ implies $\cM^*\approx\hat{\cM}^*$. 
\end{prop}

\begin{exa}[\mbox{\cite[Ex.~4 and Ex.~47]{JuPe18}}]\label{E-DualUnif}
Let $\cU_{k}(E)=(E,\rho)$ be the \emph{uniform $q$-matroid of rank $k$}, that is, $\rho(V)=\min\{k,\dim V\}$ for all $V\in\cV(E)$.
Then $\cU_{k}(E)^*=\cU_{\dim E-k}(E)$.
\end{exa}

The rest of this section is devoted to \qPM{}s induced by rank-metric codes. 
This will provide us with a large class of \qPM{}s.
We start with collecting some basic properties of codes in $\Fnm$.
As usual, we endow $\Fnm$ with the rank-metric, defined as $\dd(A,B)=\rk(A-B)$.

We only consider \emph{linear rank-metric codes}, that is, subspaces of the metric space $(\Fnm,\dd)$.
The following is standard knowledge in the theory of rank-metric codes.
For $V\leq\F^n$ denote by $V^\perp\leq\F^n$ the orthogonal space with respect to the standard dot product.

\begin{rem}\label{D-RMCBasics}
Let  $\cC\leq \Fnm$ be a rank-metric code. 
\begin{alphalist}
\item The \emph{rank distance} of~$\cC$ is defined as $\drk(\cC)=\min\{\rk(M)\mid M\in\cC\setminus0\}$.
If $d=\drk(\cC)$, then $\dim(\cC)\leq \max\{m,n\}(\min\{m,n\}-d+1)$, which is known as the \emph{Singleton bound}.
        If  $\dim(\cC)=\max\{m,n\}(\min\{m,n\}-d+1)$, then $\cC$ is called an \emph{MRD code}.
\item The \emph{dual code} of~$\cC$ is defined as $\cC^\perp=\{M\in\Fnm\mid \tr(MN\T)=0\text{ for all }N\in\cC\}$, where 
        $\tr(\cdot)$ denotes the trace of the given matrix.
\item For $V\in\cV(\F^n)$ we set $\cC(V,\cc)=\{M\in\cC\mid \colsp(M)\leq V\}$,
where $\colsp(M)$ denotes the column space of~$M$.
Then $ \Fnm(V,\cc)^\perp=\Fnm(V^\perp,\cc)$ 
and~\cite[Lem.~28]{Ra16a}
\begin{align*}
  \dim\cC(V^\perp,\cc)&=\dim\cC-m\dim V+\dim\cC^\perp(V,\cc).
\end{align*}
\end{alphalist}
\end{rem}

The following  \qPM{}s induced by rank-metric codes appeared first in~\cite{GJLR19}. 
The statement in \eqref{e-rhoV} is immediate with Proposition~\ref{D-RMCBasics}(c).

\begin{defiprop}[\mbox{\cite[Thm.~5.3]{GJLR19}}]\label{P-RMCMatroid}
For a  nonzero rank-metric code $\cC\leq\F^{n\times m}$ define 
\[
   \rho_\cc:\cV(\F^n)\longrightarrow \Q_{\geq0},\quad V\longmapsto \frac{\dim\cC-\dim\cC(V^\perp,\cc)}{m}.
\]
Then $\rho_\cc$ is a $q$-rank function with denominator~$m$. 
The $q$-polymatroid $\cM_\cc(\cC):=(\F^n,\rho_\cc)$ is called the 
\emph{(column) polymatroid} of~$\cC$.
Its rank is $\dim\cC/m$.
The rank function satisfies
\begin{equation}\label{e-rhoV}
   \rho_\cc(V)=\dim V-\frac{1}{m}\dim\cC^\perp(V,\cc).
\end{equation}
\end{defiprop}

Analogously we can define the row polymatroid of the rank-metric code, which then has denominator~$n$.
It is of course the same as the column polymatroid of the transposed code, and thus it suffices to consider 
column polymatroids.

The denominator~$m$ is in general not principal; see for instance~(a) of \cref{P-MRDCodes} below.

The rank distances of a code and its dual are closely related to the \qPM. The following is immediate with Proposition~\ref{P-RMCMatroid} and 
also  appeared already in \cite[Prop.~6.2]{GJLR19}, \cite[Lem.~30]{JuPe18}, and \cite[Rem.~3.8]{GLJ21}.

\begin{rem}\label{R-ddperp}
Let $\cC\leq\F^{n\times m}$ be a nonzero rank-metric code with rank-distance~$d$ and let $d^\perp$ be the rank distance of 
$\cC^\perp$. Then for any $V\in\cV(\F^n)$ we have
\[
   \rho_\cc(V)=\begin{cases}\dim V,&\text{ if }\dim V< d^\perp,\\[1ex]
        \frac{\dim \cC}{m},&\text{ if }\dim V>n-d.\end{cases}
\]
\end{rem}

For MRD codes we can give more detailed information about the associated \qPM{}s.

\begin{prop}[\mbox{\cite[Cor.~6.6]{GJLR19}, \cite[Thm.~3.12]{GLJ21}}]\label{P-MRDCodes}
Let $\cC\leq\Fnm$ be an MRD code with $\drk(\cC)=d$. 
\begin{alphalist}
\item Let $n\leq m$. Then $\cM_\cc(\cC)=\cU_{n-d+1}(\F^n)$, that is, $\cM_\cc(\cC)$ is the uniform $q$-matroid of rank $n-d+1$.
         Thus $\cM_\cc(\cC)$ depends only on $(n,d,|\F|)$.
\item Let $n\geq m$. Then $\cM_\cc(\cC)$ satisfies
         \[
             \rho_\cc(V)=
             \left\{\begin{array}{cl} \dim V,&\text{if }\dim V\leq m-d+1,\\[.5ex] \frac{n(m-d+1)}{m},&\text{if }\dim V\geq n-d+1,
             \end{array}\right.
         \]
         and $\rho_\cc(V)\geq \max\{1,\,(\dim V)/m\}(m-d+1)$ if $\dim V\in[m-d+2,\,n-d]$.
         Thus, if $m=n-1$, then $\cM_\cc(\cC)$ depends only on $(n,d,|\F|)$.
\end{alphalist}
\end{prop}

Equivalence of codes, in the usual sense, translates into equivalence of the associated \qPM{}s.

\begin{prop}[\mbox{\cite[Prop.~3.5]{GLJ21} and \cite[Prop.~6.7]{GJLR19}}]\label{P-EquivCodeMatroid}
Let  $\cC,\,\cC'\leq\Fnm$ be rank-metric codes.
\begin{alphalist}
\item Suppose $\cC,\,\cC'$ are \emph{equivalent}, i.e., $\cC'=X\cC Y:=\{XMY\mid M\in\cC\}$ for some 
        $X\in\GL_n(\F)$ and $Y\in\GL_m(\F)$.
       Then $\cM_\cc(\cC)$ and $\cM_\cc(\cC')$ are equivalent via $\beta\in\Hom_{\F}(\F^n,\F^n)$
        given by $x\mapsto (X\T)^{-1}x$.
\item Let $n=m$ and suppose $\cC,\,\cC'$ are \emph{transposition-equivalent}, that is, $\cC'=X\cC\T Y$ for some
       $X,Y\in\GL_n(\F)$, and where $\cC\T=\{M\T\mid M\in\cC\}$.
         Then $\cM_\cc(\cC\T)$ and $\cM_\cc(\cC')$ are equivalent via~$\beta$, where~$\beta$ is as in~(a).
\end{alphalist}
\end{prop}

Occasionally, we will consider $\F_{q^m}$-linear rank-metric codes, which we introduce as follows.
Recall that $\F=\F_q$.
Let~$\omega$ be a primitive element of the field extension~$\F_{q^m}$ and $\psi:\F_{q^m}\longrightarrow\F_q^m$ be the coordinate map with respect to the basis $(1,\omega,\ldots,\omega^{m-1})$.
Extending~$\psi$ entry-wise, we obtain, for any~$n$, an isomorphism 
$\Psi:\F_{q^m}^n\longrightarrow\F_q^{n\times m}$ that maps
$(c_1,\,\cdots,\, c_n)$ to the matrix with rows $\psi(c_1),\, \ldots,\, \psi(c_n)$.
Let  $f=x^m-\sum_{i=0}^{m-1}f_ix^i\in\F[x]$ be the minimal polynomial of~$\omega$ over~$\F$ and
\begin{equation}\label{e-Deltaf}
    \Delta_f=
    \begin{pmatrix} &1& & \\ & &\ddots & \\ & & & 1\\ f_0&f_1&\cdots&f_{m-1}\end{pmatrix}\in\GL_m(\F_q)
\end{equation}
be the companion matrix of~$f$.

Given a code $\cC\leq\Fnm$, we can now describe linearity of the code $\Psi^{-1}(\cC)\subseteq\F_{q^m}^n$ over $\F_{q^m}$ or a subfield thereof as follows; see also~\cite[Sec.~3]{GLJ21}.
The matrix~$X$ below simply allows for arbitrary bases of $\F_{q^m}$ over~$\F_q$.

\begin{defi}\label{D-Fqn-linear}
Let $\cC\leq\F_q^{n\times m}$ be a rank-metric code (hence an $\F_q$-linear subspace).
Let $s$ be a divisor of~$m$ and set $M=(q^m-1)/(q^s-1)$.
Then $\cC$ is \emph{right $\F_{q^s}$-linear} if there exists an $X\in\GL_m(\F_q)$ such that
the code $\cC X$ is invariant under right multiplication by~$\Delta_f^M$.
\emph{Left linearity} over $\F_{q^n}$ and its subfields is defined analogously.
\end{defi}

Obviously, the qualifiers left/right are needed only in the case where $\F_{q^s}$ is a subfield of both $\F_{q^n}$ and $\F_{q^m}$. 
The following is now easy to verify; see also~\cite[Sec.~3]{GLJ21}.

\begin{rem}\label{R-rhoFqnlinear}
Let~$\F_{q^s}$ be a subfield of $\F_{q^m}$ and $\cC\leq\F_q^{n\times m}$  be a right $\F_{q^s}$-linear rank-metric code.
Then $\mu=m/s$ is a denominator of the column polymatroid $\cM_\cc(\cC)$.
In particular, for $s=m$ the  poly\-matroid $\cM_\cc(\cC)$ is a $q$-matroid.
These are exactly the $q$-matroids studied in \cite{JuPe18}.
\end{rem}

It should be noted that $\cM_\cc(\cC)$ may be a $q$-matroid even if~$\cC$ is not right 
$\F_{q^m}$-linear. 
Indeed, in \cref{P-MRDCodes}(a) we saw already that the polymatroid associated to an MRD code in $\Fnm$ is a $q$-matroid
if $n\leq m$.
Furthermore, \cite[Thm.~5.5]{GLJ21} shows that if $\cC$ induces a $q$-matroid, then so does every shortening and puncturing of~$\cC$.

Duality of \qPM{}s (see \cref{T-DualqPM}) corresponds to duality of codes.

\begin{theo}[\mbox{\cite[Thm.~8.1]{GJLR19}}]\label{T-TraceDual}
Let $\cC\leq\F^{n\times m}$ be a rank-metric code and $\cC^\perp\leq\F^{n\times m}$ be its dual.
Then $\cM_\cc(\cC)^*=\cM_\cc(\cC^\perp)$, where $\cM_\cc(\cC)^*$ is the dual of $\cM_\cc(\cC)$ w.r.t.\ the standard dot product on~$\F^n$.
\end{theo}

Another instance of the interplay between polymatroid duality and rank-metric codes has been presented in 
\cite[Thms.~5.3 and 5.5]{GLJ21}.
Therein,  it is shown that contraction and deletion of \qPM{}s are mutually dual and 
correspond to shortening and puncturing of rank-metric codes.

We close this section with the following example of two MRD codes whose associated column polymatroids are not equivalent.
The parameters are as in \cref{P-MRDCodes}(b), where the interval $[m-d+2,\,n-d]$ is not empty.
We will return to this example later when discussing independent spaces.

\begin{exa}[\mbox{\cite[Ex.~3.15]{GLJ21}}]\label{E-RowMatMRD}
In $\F_2^{5\times 2}$ consider the codes $\cC_1=\subspace{A_1,\dots,A_5}$ and $\cC_2=\subspace{B_1,\ldots,B_5}$, where
\begin{align*}
   &A_1=\begin{pmatrix}1&1\\1&0\\0&0\\1&0\\0&0\end{pmatrix},\
    A_2=\begin{pmatrix}1&1\\1&1\\1&0\\0&1\\0&0\end{pmatrix},\
    A_3=\begin{pmatrix}0&0\\0&0\\1&1\\0&0\\0&1\end{pmatrix},\
    A_4=\begin{pmatrix}0&0\\0&1\\0&0\\0&0\\1&1\end{pmatrix},\
    A_5=\begin{pmatrix}1&0\\0&1\\1&1\\0&0\\0&1\end{pmatrix},\\[2ex]
   &B_1=\begin{pmatrix}1&0\\0&1\\0&0\\0&0\\0&0\end{pmatrix},\
     B_2=\begin{pmatrix}0&0\\1&0\\0&1\\0&0\\0&0\end{pmatrix},\
     B_3=\begin{pmatrix}0&0\\0&0\\1&0\\0&1\\0&0\end{pmatrix},\
     B_4=\begin{pmatrix}0&0\\0&0\\0&0\\1&0\\0&1\end{pmatrix},\
     B_5=\begin{pmatrix}0&1\\0&0\\0&1\\0&0\\1&0\end{pmatrix}.
\end{align*}
Both codes are MRD with rank distance~$d=2$, and~$\cC_2$ is actually a ($\F_{2^5}$-linear) Gabidulin code.
Consider the \qPM{}s $\cM_\cc(\cC_1)=(\F^5,\rho_\cc^1)$ and $\cM_\cc(\cC_2)=(\F^5,\rho_\cc^2)$.
From \cref{P-MRDCodes}(b) we know that $\rho_\cc^1(V)=\rho_\cc^2(V)=\dim V$ for $\dim V\leq 1$ and 
$\rho_\cc^1(V)=\rho_\cc^2(V)=5/2$ if $\dim V\geq 4$.
As for the $2$-dimensional subspaces of~$\F_2^5$, it turns out that the map $\rho_\cc^1$ assumes the value~$1$ exactly once and the values 
$3/2$ and $2$ exactly 28 and 126 times, respectively, whereas~$\rho_\cc^2$ assumes the values~$3/2$ and~$2$ exactly 31 and 124 times, respectively, and never takes the value~$1$.
Similar differences occur for the $3$-dimensional subspaces.
Thus $\cM_\cc(\cC_1)$ and $\cM_\cc(\cC_2)$ are not equivalent.
\end{exa}

\section{Independent Spaces} \label{S-Indep}
We now return to general \qPM{}s and introduce independent spaces. 
We show that the collection of independent spaces satisfies properties analogous to those of independent spaces in 
$q$-matroids. 
However, different from the latter, they do not fully determine the \qPM.
Only if we also take the rank values of the independent spaces into account, can we fully recover the \qPM. 
This will be dealt with in the next section.

Considering the theory of classical matroids and $q$-matroids, one may be inclined to declare a space~$V$ in a \qPM{} $(E,\rho)$
 independent if $\rho(V)=\dim V$.
While this is indeed the right notion for $q$-matroids, it turns out to be too restrictive for \qPM{}s:
in many \qPM{}s the only subspace satisfying  $\rho(V)=\dim V$  is the zero space.
Nonetheless, the property $\rho(V)=\dim V$ turns out to play a conceptual role (see also~\cite{BCIJ21}), and we will 
return to it in Section~\ref{S-SpSp}, where we will call such spaces strongly independent.

The following definition of independence is inspired by \cite[Cor.~11.1.2]{Ox11}, which deals with classical polymatroids.

\begin{defi}\label{D-Indep}
Let $\cM=(E,\rho)$ be a \qPM{} with denominator~$\mu$ (which need not be principal). 
A space $I\in\cV(E)$ is called \emph{$\mu$-independent} if 
\[
     \rho(J)\geq \frac{\dim J}{\mu} \text{ for all subspaces }J\leq I.
\]
$I$ is called \emph{$\mu$-dependent} if it is not $\mu$-independent.
A \emph{$\mu$-circuit} is a $\mu$-dependent space for which all proper subspaces are $\mu$-independent.
A $1$-dimensional $\mu$-dependent space is called a \emph{$\mu$-loop}.
We define $\cI_\mu=\cI_\mu(\cM)=\{I\in\cV(E)\mid I\text{ is $\mu$-independent}\}$.
If~$\mu$ is the principal denominator of~$\cM$, we may skip the quantifier~$\mu$ and simply use independent, 
dependent, loop, circuit, and~$\cI$.
\end{defi}

Clearly, if $\hat{\mu}$ is the principal denominator of~$\cM$, then $\hat{\mu}$-independence implies $\mu$-independence for any 
denominator~$\mu$ of~$\cM$; see \cref{R-PrDenominator}.

\begin{rem}\label{R-IndepMatroid}
Let  $(E,\rho)$ be a \qPM{}. Then for all $V\in\cV(E)$
\begin{equation}\label{e-rhoVdimV}
     \rho(V)=\dim V\Longrightarrow \rho(W)=\dim W\text{ for all }W\leq V.
\end{equation}
Indeed, writing $V=W\oplus Z$ for some complement~$Z$ of~$W$, we obtain from 
submodularity
$\dim V=\rho(V)\leq\rho(W)+\rho(Z)\leq\dim W+\dim Z=\dim V$, and thus we have equality everywhere.
As a consequence, the condition $\rho(V)=\dim V$ implies $\mu$-independence for every denominator~$\mu$, and 
for $q$-matroids our notion of independence (which is $1$-independence) coincides with independence as 
defined in \cite[Def.~2]{JuPe18}, that is: 
$V \text{ is $1$-independent }\Longleftrightarrow \rho(V)=\dim V$.
\end{rem}

We continue with discussing basis properties of independent spaces. 
A crucial difference to $q$-matroids is the following:
While for $q$-matroids a space is independent iff its rank value assumes the maximal possible value, this
is not the case for \qPM{}s.
More generally, for \qPM{}s independence is not characterized by the rank value of the given space; see the examples below.
In this context we also would like to point out that the inequality $\rho(I)\geq\dim I/\mu$ is not preserved under taking subspaces, 
which is why the condition for subspaces is built into our definition.

\begin{exa}\label{E-IndSpacesRanks}
\begin{alphalist}
\item Let $\F=\F_2$ and consider the code $\cC\leq\F^{3\times 3}$ generated by
        \[
          \begin{pmatrix}0&1&0\\0&0&1\\0&0&1\end{pmatrix},\   
          \begin{pmatrix}0&1&1\\0&0&0\\0&0&1\end{pmatrix},\ 
           \begin{pmatrix}0&1&1\\1&0&0\\0&1&0\end{pmatrix}.
        \]
        Let $\cM=\cM_\cc(\cC)=(\F^3,\rho_\cc)$ be the associated column polymatroid.
        Then for all $V\in\cV(\F^3)\setminus0$
        \[
            \rho_\cc(V)=\left\{\begin{array}{cl} 2/3 &\text{ if $\dim V=1$ or $V=\subspace{e_1+e_2,\,e_3}$,}\\[.4ex]
                                                                   1 &\text{ otherwise.}\end{array}\right.
        \]
        This shows that $3$ is the principal denominator and $\cI(\cM)=\cV(\F^3)$, that is, all spaces are independent. 
        In particular, all $2$-dimensional spaces are independent, even though they do not assume the same rank value.
\item Dependent spaces may have a larger rank value than independent spaces of the same dimension.
For instance, let $\F=\F_2$ and $\cC\leq\F^{5\times 3}$ be the code generated by the standard basis matrices
$E_{11},\,E_{12},\,E_{23},\,E_{32},\,E_{41},\,E_{42}$.
In the column polymatroid $\cM_\cc(\cC)=(\F^5,\rho_\cc)$ the subspace $I=\subspace{e_2,e_3}$ is independent with $\rho_\cc(I)=2/3$,
and the subspace $V=\subspace{e_1+e_2,e_5}$ is dependent with $\rho_\cc(V)=1$.
\end{alphalist}
\end{exa}

Let us consider the independent spaces of some \qPM{}s from the previous section.

\begin{exa}\label{E-IndSpaces}
\begin{alphalist}
\item Consider the uniform $q$-matroid $\cU=\cU_{k}(E)$ from \cref{E-DualUnif}.
         Then $\cI_1(\cU)=\{V\in\cV(E)\mid \dim(V)\leq k\}$.
\item Let $m\geq n$ and $\cC\leq\Fnm$ be an MRD code with rank distance~$d$. 
         \cref{P-MRDCodes}(a) yields $\cM_\cc(\cC)=\cU_{n-d+1}(\F^n)$  and thus
         $\cI_1(\cM_\cc(\cC))=\{V\in\cV(\F^n)\mid \dim V\leq n-d+1\}$ by~(a).
\item Let $m\leq n$ and $\cC\leq\Fnm$ be an MRD code with rank distance~$d$. 
         With the aid of \cref{P-MRDCodes}(b) one verifies that
         $\cI_m(\cM_\cc(\cC))=\cV(\F^n)$, that is, every subspace is $m$-independent.
         In particular, the two non-equivalent \qPM{}s $\cM_\cc(\cC_1)$ and $\cM_\cc(\cC_2)$ 
         in \cref{E-RowMatMRD} trivially have the same collection of $2$-independent spaces 
         (and~$2$ is the principal denominator). 
         This shows that the collection of independent spaces does not determine the \qPM.
\item  Part~(c) shows another striking difference to $q$-matroids. If all spaces of a $q$-matroid $\cM=(E,\rho)$ 
         are independent (i.e., $1$-independent), then~$\cM$ is the uniform matroid $\cU_{n}(E)$, where $n=\dim E$.
         If furthermore $\cM=\cM_\cc(\cC)$ for some right $\F_{q^m}$-linear code $\cC\leq\Fnm$, then 
         $\cC=\Fnm$. This follows immediately from \cref{R-IndepMatroid} together with~\eqref{e-rhoV}.
\end{alphalist}
\end{exa}

We continue with some basic facts.

\begin{rem}\label{R-IndJuPe}
Let $\cM=(E,\rho)$ be a \qPM{} with denominator~$\mu$.
\begin{alphalist}
\item The zero subspace of~$E$ is $\mu$-independent.
\item Every dependent space~$V$ contains a circuit: take any subspace $W$ of~$V$ of smallest dimension satisfying
         $\rho(W)<\dim W/\mu$ (which clearly exists).
\item Let $V\in\cV(E)$ be a $\mu$-circuit. Then $\mu\rho(V)=\dim V-1=\mu\rho(W)$ for all hyperplanes~$W$ in~$V$. 
       Indeed, independence of~$W$ along with (R2) tells us that
       \[
           \dim V-1=\dim W \leq\mu\rho(W)\leq\mu\rho(V)<\dim V.
        \]
        Thus we have equality since $\mu\rho$ takes integer values. 
        \cref{E-FqslinearCircuits} below shows that not every subspace~$V$ satisfying $\mu\rho(V)=\dim V-1=\mu\rho(W)$ for all its      
        hyperplanes~$W$ is a circuit. 
\item Part~(c) implies that if~$\dim V=1$, then $V$ is a $\mu$-loop iff $\rho(V)=0$.
         As a consequence, loops do not depend on the choice of denominator. 
        On the other hand, if $\mu_1<\mu_2$ are distinct denominators of~$\cM$, then the $\mu_2$-circuits with positive rank 
        value are distinct from the $\mu_1$-circuits.
        But every $\mu_2$-circuit is also $\mu_1$-dependent and thus contains a $\mu_1$-circuit by (b).
\item Let $x_1,\ldots,x_t\in E$ be linearly independent vectors such that $\subspace{x_i}$ is a $\mu$-loop for all $i\in[t]$.
        Then $\rho(\subspace{x_1}+\ldots+\subspace{x_t})=0$. For $t=2$ this is a consequence of the submodularity~(R3) because
        $\rho(\subspace{x_1}+\subspace{x_2})=\rho(\subspace{x_1}+\subspace{x_2})+\rho(\subspace{x_1}\cap\subspace{x_2}) 
        \leq\rho(\subspace{x_1})+\rho(\subspace{x_2})$ (see also \cite[Lem.~11]{JuPe18} for $q$-matroids), and the general case follows similarly via induction.
\end{alphalist}
\end{rem}

\begin{exa}\label{E-FqslinearCircuits}
Let $\F=\F_2$.
Consider the primitive polynomial $f=x^4+x+1\in\F[x]$ and let $U=\Delta_f^5$, where $\Delta_f$ is defined as in~\eqref{e-Deltaf}.
Thus $U\in\GL_4(\F)$.
Let
\[
  A_1=\begin{pmatrix}0 & 1 & 1 & 0 \\
0 & 0 & 0 & 0 \\
0 & 0 & 0 & 0 \\
1 & 0 & 0 & 1 \\
0 & 0 & 1 & 1 \\
1 & 0 & 0 & 1\end{pmatrix},\quad
  A_2=\begin{pmatrix}1 & 0 & 0 & 0 \\
0 & 0 & 0 & 0 \\
1 & 0 & 0 & 0 \\
1 & 0 & 0 & 1 \\
0 & 1 & 1 & 1 \\
0 & 0 & 0 & 1\end{pmatrix},\quad
  A_3=\begin{pmatrix}1 & 1 & 0 & 1 \\
0 & 0 & 1 & 0 \\
0 & 0 & 1 & 0 \\
1 & 0 & 0 & 0 \\
0 & 0 & 0 & 1 \\
1 & 0 & 1 & 0\end{pmatrix}
\]
and define $\cC=\subspace{A_1,\,A_2,\,A_3,\,A_1U,\,A_2U,\,A_3U}$. 
Then~$\cC$ is a right $\F_{2^2}$-linear rank-metric code of dimension~$6$ and rank distance~$3$ 
(see \cref{D-Fqn-linear} for right $\F_{2^2}$-linearity).
The principal denominator of~$\cM_\cc(\cC)$ is $\mu=2$.
There exist 497 $\mu$-circuits, one of which has dimension~$1$ and all others have dimension~$4$. 
An additional 169 spaces~$V$ satisfy $\mu\rho_\cc(V)=\dim V-1$ (all of them have dimension~$2,3$, or~$4$),
and 97 of them also satisfy $\mu\rho_\cc(W)=\dim V-1$ for all its hyperplanes~$W$.
\end{exa}

Independence behaves well under scaling-equivalence if the denominator is taken into account.

\begin{rem}\label{R-Rescaling}
Let $\dim E_1=\dim E_2$ and $\cM_i=(E_i,\rho_i), i=1,2$, be \qPM{}s with principal denominators $\mu_i$.
Suppose $\cM_1$ and~$\cM_2$ are scaling-equivalent, say $\rho_2(\alpha(V))=a\rho_1(V)$  for all $V\in\cV(E_1)$, where $a\in\Q_{>0}$ and $\alpha:E_1\longrightarrow E_2$ an isomorphism.
Then $a^{-1}\mu_1\rho_2(\alpha(V))=\mu_1\rho_1(V)\in\N$ and thus $a^{-1}\mu_1$ is a denominator of
$\cM_2$. 
Hence $a^{-1}\mu_1=k\mu_2$ for some $k\in\N$; see Remark~\ref{R-PrDenominator}.
Similarly, $a\mu_2=\hat{k}\mu_1$ for some $\hat{k}\in\N$.
Thus $k=\hat{k}=1$, and hence $a\mu_2=\mu_1$. 
Now we have 
$\mu_2\rho_2(\alpha(V))=\mu_1\rho_1(V)$ for all $V\in\cV(E)$ and therefore
\begin{align*}
  V\text{ is $\mu_1$-independent in  $\cM_1$}&\Longleftrightarrow
   \mu_1\rho_1(W)\geq\dim W\text{ for all subspaces }W\leq V\\
    &\Longleftrightarrow\mu_2\rho_2(\alpha(W))\geq\dim\alpha(W) \text{ for all subspaces }\alpha(W)\leq\alpha(V)\\
     &\Longleftrightarrow \alpha(V)\text{ is $\mu_2$-independent in  $\cM_2$}.
\end{align*}
\end{rem}

Before we continue with our study of independent spaces, we briefly focus on \qPM{}s induced by rank-metric codes and 
discuss the relation of (in-)dependent spaces and code properties.
Obviously, the relation depends on the chosen denominator.

\begin{rem}\label{R-RMCMatroidIndSp}
Let $\cC\leq\Fnm$ be a rank-metric code with column polymatroid $\cM=(\F^n,\rho_\cc)$.
\begin{alphalist}
\item The most obvious information about the code is contained in \cref{R-ddperp}: the distance~$d^\perp$ of the dual code is 
         the largest integer $\ell$ for which all $\ell$-dimensional subspaces~$V$ satisfy $\rho_\cc(V)=\dim V$. From
         \cref{R-IndepMatroid} we know that all these subspaces are $\mu$-independent for every denominator~$\mu$ of~$\cM$.
\item \cref{D-Indep} shows that the condition for independence relaxes with increasing denominator.
         For this reason there may be few dependent spaces if the principal denominator of~$\cM$ is~$m$.
         To make this more precise, let us consider the circuits for a given denominator~$\mu$.
         From \cref{R-IndJuPe}(c) we know that if $V$ is a $\mu$-circuit of dimension~$v$, then $\mu\rho_\cc(V)=v-1$. 
         The definition of $\rho_\cc$ thus implies
         \[
           V\text{ is a $\mu$-circuit}\Longrightarrow \dim\cC(V^\perp,\cc)=\dim\cC-\frac{mv}{\mu}+1
          \]
          (recall from \cref{R-PrDenominator} that $\mu$ is a divisor of~$m$).
          The right hand side means that~$\cC$ is equivalent to a code
           \[
             \tilde{\cC}=\bigg\langle
             \begin{pmatrix}A_{1,1}\\0\end{pmatrix},\ldots,\begin{pmatrix}A_{1,k-\gamma+1}\\0\end{pmatrix},
             \begin{pmatrix}A_{1,k-\gamma}\\A_{2,k-\gamma}\end{pmatrix},\ldots,\begin{pmatrix}A_{1,k}\\A_{2,k}\end{pmatrix}\bigg\rangle,
          \]
          where $k=\dim\cC,\,\gamma=mv/\mu$, and $A_{1,j}\in\F^{(n-v)\times m},\,A_{2,j}\in\F^{v\times m}$.
          In particular, if~$\cM$ contains an $m$-loop, then $\gamma=1$ and~$\cC$ is equivalent to a row degenerate code: 
          the last row of all matrices in~$\tilde{\cC}$ is zero.
          More generally, for any fixed dimension~$v$, the existence of a $v$-dimensional $\mu$-circuit becomes more restrictive 
           with increasing~$\mu$.
          \end{alphalist}
\end{rem}

We now return to general \qPM{}s.
In order to derive the main result about the collection of $\mu$-independent spaces, we will make use of an auxiliary $q$-matroid.
The following construction mimics the corresponding one in \cite[Prop.~11.1.7]{Ox11} for classical polymatroids.

\begin{theo}\label{T-AuxMatroid}
Let~$\cM=(E,\rho)$ be a \qPM{} with denominator~$\mu$. Define the map
\[
  r_{\rho,\mu}:\cV(E)\longrightarrow\N_0,\quad V\longmapsto \min\{\mu\rho(W)+\dim V-\dim W\mid W\leq V\}.
\]
Then $\cZ:=\cZ_{\cM,\mu}:=(E, r_{\rho,\mu})$ is a $q$-matroid, and the independent spaces of $\cZ$ coincide with the $\mu$-independent spaces of~$\cM$, i.e.,
\[
     \cI_\mu(\cM)=\cI(\cZ)=\{I\in\cV(E)\mid r_{\rho,\mu}(I)=\dim I\}.
\]
\end{theo}

\begin{proof}
We make use of the induced integer $\rho$-function $\tau=\mu\rho$.
Thus $\tau(V)=\mu\rho(V)\leq\mu\dim V$ for all $V\in\cV(E)$.
Clearly the map $r:=r_{\rho,\mu}$ takes integer values. 
We now verify (R1) -- (R3) of \cref{D-PMatroid} for $r$.
\\[.5ex]
(R1) Clearly $r(V)\geq0$ for all~$V$. Furthermore, $r(V)\leq \tau(0)+\dim(V)-\dim(0)=\dim V$.
\\[.5ex]
(R2)
Let $V\leq V'$. 
Without loss of generality we may assume $\dim V'=\dim V+1$ and thus $V'=V\oplus\subspace{x}$ for some $x\in E$.
Assume by contradiction that $r(V)>r(V')$. 
Then there exists $W'\leq V'$ such that
\begin{equation}\label{e-rineq}
   \tau(W')+\dim V'-\dim W'<\tau(W)+\dim V-\dim W\ \text{ for all }W\leq V.
\end{equation}
Clearly $W'\not\leq V$ and thus we may write $W'=X\oplus\subspace{y}$ for some $X\leq V$ and $y\not\in V$.
Then $\dim X=\dim V-\dim V'+\dim W'$ and~\eqref{e-rineq} leads to
\[
     \tau(W)\!-\!\dim W>\tau(W')\!+\!\dim V'\!-\!\dim W'\!-\!\dim V=\tau(W')\!-\!\dim X\text{ for all }W\leq V.
\]
Choosing $W=X$, we arrive at $\tau(X)>\tau(W')$ and thus $\rho(X)>\rho(W')$.
Since $X\leq W'$ this contradicts that~$\rho$ is a rank function.
 All of this establishes (R2) for the map~$r$.
 \\[.5ex]
(R3) Let $V,V'\in\cV(E)$. Choose  $W\leq V,\,W'\leq V'$ such that
\[
   r(V)=\tau(W)+\dim V-\dim W\ \text{ and }\ r(V')=\tau(W')+\dim V'-\dim W'.
\]
Then $W+W'\leq V+ V'$ and $W\cap W'\leq V\cap V'$ and therefore
\begin{align*}
  r(V\!+\!V')+r(V\cap V')&\leq\tau(W+W')+\dim(V+V')-\dim(W+W')\\
       &\quad\ +\tau(W\cap W')+\dim(V\cap V')-\dim(W\cap W')\\
     &=\tau(W\!+\!W')+\tau(W\!\cap\! W')+\dim V-\dim W+\dim V'-\dim W'\\
     &\leq \tau(W)+\tau(W')+\dim V-\dim W+\dim V'-\dim W'\\
     &=r(V)+r(V'),
\end{align*}
where the second inequality follows from~(R3) for~$\rho$.
This establishes~(R3) for the map~$r$.
\\[.5ex]
It remains to investigate the $\mu$-independent spaces. 
From \cref{D-Indep},~(R1), and \cref{R-IndepMatroid} we obtain 
immediately 
\begin{align*}
  \text{$V$ is $\mu$-independent }&\Longleftrightarrow \tau(W)\geq\dim W\text{ for all }W\leq V\\
    &\Longleftrightarrow \tau(W)+\dim V-\dim W\geq\dim V\text{ for all }W\leq V\\
    &\Longleftrightarrow r(V)\geq\dim V\\
    &\Longleftrightarrow r(V)=\dim V,
\end{align*}
and this establishes the stated result.
\end{proof}

As the next example shows, if~$\cM$ is a $q$-matroid, then it coincides with its auxiliary $q$-matroid if we choose the principal 
denominator $\mu=1$.

\begin{exa}\label{E-Auxmu1}
Let~$\cM=(E,\rho)$ be a $q$-matroid, thus $\rho$ takes only integer values. 
\begin{alphalist}
\item Fix $\mu=1$.
        Then $r_{\rho,1}(V)=\min\{\rho(W)+\dim V-\dim W\mid W\leq V\}$ for $V\in\cV(E)$.
        We show now that $r_{\rho,1}=\rho$. Fix $V\in\cV(E)$.
        Choosing $W=V$ we obtain $r_{\rho,1}(V)\leq \rho(V)$.
        On the other hand, for every $W\leq V$ there exists $Z\leq V$ such that $W\oplus Z=V$. 
        Thus submodularity (R3) applied to~$\rho$ yields 
        \[
                  \rho(V)=\rho(W+Z)\leq\rho(W)+\rho(Z)\leq\rho(W)+\dim Z
                  =\rho(W)+\dim V-\dim W.
        \]
        This shows $\rho(V)\leq r_{\rho,1}(V)$.
\item If we choose $\mu>1$, then there is in general no obvious relation between~$\rho$ and~$r_{\rho,\mu}$. 
        Consider for example the following $q$-matroid. Let $n\geq3$ and fix a $2$-dimensional subspace $X\in\cV(\F^n)$.
        Set $\rho(X)=1$ and $\rho(V)=\min\{\dim V,\,2\}$ for $V\neq X$.
        One can check straightforwardly that $\cM=(\F^n,\rho)$ is a $q$-matroid (this also follows from \cite[Prop.~4.7]{GLJ21}).
        Choosing $\mu=2$,  one obtains $r_{\rho,2}=\min\{\dim V,4\}$, and thus
        the $q$-matroids $\cM$ and $\cZ_{\cM,2}$ 
        are not equivalent.
        Furthermore, $\cZ_{\cM,2}=\cZ_{\cU,2}$, where $\cU$ is the uniform $q$-matroid $\cU_2(\F^n)$, and therefore
         the auxiliary $q$-matroid $\cZ_{\cM,\mu}$ of a \qPM~$\cM$ does not uniquely determine~$\cM$ (even if one specifies the denominator).
\end{alphalist}
\end{exa}

\cref{T-AuxMatroid} shows that the $\mu$-independent spaces of the \qPM{} $\cM$ coincide with 
the independent spaces of the auxiliary $q$-matroid $\cZ_{\cM,\mu}$.
Therefore, all properties of independent spaces of $q$-matroids that do not involve the value of the rank function hold true 
for \qPM{}s as well.
Such properties have been derived in \cite[Thm.~8]{JuPe18}.
Before formulating our result we cast the following important notions.

\begin{defi}\label{D-MaxInd}
Let~$\cM=(E,\rho)$ be a \qPM{} with denominator~$\mu$.
For $V\in\cV(E)$ we define 
\[
     \cI_\mu(V)=\{I\in\cI_\mu(\cM)\mid I\leq V\}.
\]
A subspace $\hat{I}\in\cI_\mu(V)$ is said to be a \emph{$\mu$-basis} of~$V$ 
if $\dim\hat{I}=\max\{\dim I\mid I\in\cI_\mu(V)\}$.
We denote by  $\cB_\mu(V)$ the set of all $\mu$-bases of~$V$.
The $\mu$-bases of~$E$ are called the \emph{$\mu$-bases} of~$\cM$.
\end{defi}

A $\mu$-basis of~$V$ is thus a maximal-dimensional $\mu$-independent subspace of~$V$. Their rank values will be discussed 
in the next section.
Note that the sets $\cI_\mu(V)$ and $\cB_\mu(V)$ are non-empty for every $V\in\cV(E)$ since clearly $\{0\}$ is $\mu$-independent.

We are now ready to present the following properties of the collection of $\mu$-independent spaces of a \qPM. 
The result is an immediate consequence of \cref{T-AuxMatroid} together with \cite[Thm.~8]{JuPe18}.

\begin{cor}\label{C-Indep}
Let~$\cM=(E,\rho)$ be a \qPM{} with denominator~$\mu$ and set $\cI_\mu:=\cI_\mu(\cM)$.
Then 
\begin{mylist}
\item[(I1)\hfill] $\cI_\mu\neq\emptyset$, in fact $\{0\}\in\cI_{\mu}$.
\item[(I2)\hfill] If $I\in\cI_\mu$ and $J\leq I$, then $J\in\cI_\mu$.
\item[(I3)\hfill] If $I,\,J\in\cI_\mu$ and $\dim I<\dim J$, then there exists $x\in J\setminus I$ such that 
        $I\oplus\subspace{x}\in\cI_\mu$.
\item[(I4)\hfill] Let $V,\,W\in\cV(E)$ and $I\leq V,\,J\leq W$ be $\mu$-bases of $V$ and~$W$, respectively.
       Then there exists a $\mu$-basis of $V+W$ that is contained in $I+J$.
\end{mylist}
\end{cor}

Note that~(I3) implies that for any $V\in\cV(E)$ and any $\hat{I}\in\cI_\mu(V)$ we have
\[
   \hat{I}\text{ is dimension-maximal in }\cI_\mu(V)\Longleftrightarrow\hat{I}\text{ is inclusion-maximal in }\cI_\mu(V).
\]
In other words,~$\hat{I}$ is a $\mu$-basis of~$V$ if and only if there exists no $J\in\cI_\mu(V)$ such that $\hat{I}\lneq J$.
Thus, the $\mu$-bases of~$V$ are exactly the maximal elements of the poset $(\cI_\mu(V),\leq)$.

Since the independent spaces of the auxiliary $q$-matroid $\cZ_{\cM,\mu}$ coincide with those of the \qPM{}~$\cM$, 
the same is true for the dependent spaces, circuits, and bases. 
As a consequence, any property about the collection of these spaces in $q$-matroids holds true for \qPM{}s as well.
Let us illustrate this for the dependent spaces and bases.
The following properties for $q$-matroids have been established in \cite[Thm.~63 and Lem.~66]{BCJ21} and therefore
apply to \qPM{}s as well.

\begin{cor}\label{C-DepSpacesCircuits}
Let~$\cM=(E,\rho)$ be a \qPM{} with denominator~$\mu$. 
Let $\cD_\mu$ and $\cB_\mu$ be the collection of $\mu$-dependent spaces and $\mu$-bases of~$\cM$, respectively. 
Then~$\cD_\mu$ and~$\cB_\mu$ satisfy
\begin{mylist2}
\item[(D1)\hfill] $\{0\}\not\in\cD_\mu$.
\item[(D2)\hfill] If $D_1\in\cD_\mu$ and $D_2\in\cV(E)$ such that $D_1\subseteq D_2$, then $D_2\in\cD_\mu$.
\item[(D2)\hfill] Let $D_1,\,D_2\in\cD_\mu$ be such that $D_1\cap D_2\not\in\cD_\mu$. 
        Then every subspace of $D_1+D_2$ of codimension~$1$ is in $\cD_\mu$.
\item[(B1)\hfill] $\cB_\mu\neq\emptyset$.
\item[(B2)\hfill] Let  $B_1,\,B_2\in \cB_\mu$ be such that $B_1\leq B_2$. Then $B_1=B_2$.
\item[(B3)\hfill] Let $B_1,\,B_2\in\cB_\mu$ and $A$ be a subspace of~$B_1$ of codimension~$1$ such that $B_1\cap B_2\leq A$.
                        Then there exists a $1$-dimensional subspace~$Y$ of~$B_2$ such that $A+Y\in\cB_\mu$.
\item[(B4)\hfill] Let $A_1,\,A_2\in\cV(E)$ and $I_1,\,I_2$ be maximal dimensional intersections of some members 
                        of~$\cB_{\mu}$ with~$A_1$ and $A_2$, respectively. 
                        Then there exist a maximal dimensional intersection of a member of~$\cB_{\mu}$ with $A_1+A_2$ 
                        that is contained in $I_1+I_2$.
\end{mylist2}
\end{cor}

In \cite[Cor.~65]{BCJ21} and \cite[Thm.~37]{JuPe18} it has been shown that any collection of subspaces satisfying (D1)--(D3) 
(resp.~(B1)--(B4)) is the collection of dependent spaces (resp.\ bases) of a unique $q$-matroid. 
Similar statements hold true for circuits in $q$-matroids (see \cite[Cor.~68]{BCJ21}).
The following examples illustrate that none of these properties extends to \qPM{}s -- even if we take the rank values 
into account.

\begin{exa}\label{E-MaxRho}
\begin{alphalist}
\item Consider the rank-metric codes~$\cC_1,\,\cC_2$  in \cref{E-RowMatMRD} and the associated column polymatroids 
$\cM_\cc(\cC_1)$ and $\cM_\cc(\cC_2)$.
Both have principal denominator~$2$, and in both cases every subspace of $\F^5$ is $2$-independent.
Thus the only $2$-basis of $\cM_\cc(\cC_i)$ is $\F^5$ for $i=1,2$, and in both \qPM{}s it has rank value $5/2$.
Yet, the two \qPM{}s are not equivalent. 
This shows that the bases of a \qPM{} along with their rank values do not uniquely determine the \qPM. 
Trivially, this example also shows that the circuits and dependent spaces along with their rank values do not determine the \qPM{}.
For later purposes we also note that in both $\cM_\cc(\cC_1)$ and $\cM_\cc(\cC_2)$ all 4-dimensional subspaces and plenty of 
3-dimensional subspaces have rank value~$5/2$ as well.
\item Let $\F=\F_2$ and consider the codes $\cC=\subspace{A_1,A_2,A_3},\,\cC'=\subspace{A_1,A_2,A_3'}\leq\F^{4\times 3}$, 
where 
\[
  A_1=\begin{pmatrix}0&0&0\\1&0&0\\0&1&1\\0&1&0\end{pmatrix},\quad
  A_2=\begin{pmatrix}1&0&1\\1&0&0\\0&0&0\\1&1&1\end{pmatrix},\quad
  A_3=\begin{pmatrix}0&1&0\\0&1&1\\0&1&0\\0&1&1\end{pmatrix},\quad
  A_3'=\begin{pmatrix}1&0&0\\0&0&0\\1&0&1\\1&0&1\end{pmatrix}.
\]
Consider the associated polymatroids $\cM=\cM_\cc(\cC)=(\F^4,\rho_\cc)$ and  
$\cM'=\cM_\cc(\cC')=(\F^4,\rho_\cc')$.
Both have principal denominator~$3$, and in both \qPM{}s the space~$\F^4$ is the only dependent space.
Hence~$\cM$ and~$\cM'$ share the same bases, namely all $3$-dimensional spaces.
Moreover, $\rho_\cc(V)=1=\rho_\cc'(V)$ for all bases~$V$.
Yet,~$\cM$ and~$\cM'$ are not equivalent:
in~$\cM$ the rank value~$1$ is assumed by 33 subspaces of dimension~$2$,
whereas in $\cM'$ it is assumed by 32 subspaces of dimension~$2$ (in both \qPM{}s 4 subspaces of dimension~$1$ have rank value~$1$ as well).
\end{alphalist}
\end{exa}

On the positive side, in the next section we will show that we can fully recover a \qPM{} from its independent spaces
and their rank values.
Recall from \cref{E-IndSpaces}(c) that the independent spaces alone (without their rank values) do not uniquely determine the \qPM.

\section{The Rank Function on Independent Spaces}\label{S-IndRank}
We begin by showing that for a \qPM{} the rank function is fully determined by its values on the independent spaces.
We then go on to prove that all bases of a given subspace have the same rank value, and this value coincides with the rank value of the subspace.
This result allows us to investigate whether a collection of spaces satisfying (I1)--(I4) from \cref{C-Indep} gives rise to a \qPM{} whose
collection of independent spaces is exactly the initial collection.
Since the rank value of independent spaces in a \qPM{} is not as rigid as in a $q$-matroid, we also need to specify a meaningful rank 
function on the collection of spaces. All of this results in Theorems~\ref{T-ExtIndSpaces} and \ref{T-ClosureIndep}.

\begin{theo}\label{T-rhomax}
Let~$\cM=(E,\rho)$ be a \qPM{} with denominator~$\mu$.
Then 
\[
      \rho(V)=\max\{\rho(I)\mid I\in\cI_{\mu}(V)\}\text{ for all }V\in\cV(E).
\]
\end{theo}

\begin{proof}
Let $V\in\cV(E)$. Set $\rho'(V)=\max\{\rho(I)\mid I\in\cI_{\mu}(V)\}$.
Thanks to~(R2), $\rho'(V)\leq\rho(V)$, and it remains to establish $\rho(V)\leq \rho'(V)$. 
Let $\hat{I}\in\cI_\mu(V)$ be of maximal possible dimension such that $\rho(\hat{I})=\rho'(V)$.
If~$V$ is $\mu$-independent, then $\hat{I}=V$ and we are done. Thus let~$V$ be $\mu$-dependent.
\\[.6ex]
\underline{Case 1:} $\dim\hat{I}=\dim V-1$.\\
Then $V=\hat{I}\oplus\subspace{x}$ for any $x\in V\setminus\hat{I}$ and submodularity of~$\rho$ implies $\rho(V)\leq \rho(\hat{I})+\rho(\subspace{x})$.
As before, we use the integer $\rho$-function $\tau=\mu\rho$.
Let~$s$ be minimal such that there exists an $s$-dimensional $\mu$-dependent subspace of~$V$, say~$W$. 
Such space exists by $\mu$-dependence of~$V$.
Then \cref{D-Indep} implies that  $\tau(W)<\dim W$.
By~(I2) $W$ is not contained in~$\hat{I}$ and thus $W=(W\cap\hat{I})\oplus\subspace{z}$ for some 
$z\in V\setminus0$.
Then $W\cap\hat{I}$ is $\mu$-independent and
\[
  \dim W-1=\dim(W\cap\hat{I})\leq\tau(W\cap\hat{I})\leq\tau(W)<\dim W,
\]
and hence $\tau(W\cap\hat{I})=\tau(W)$ because~$\tau$ takes integer values.
Using that $V=W+\hat{I}$, we obtain by submodularity of~$\tau$
\[
  \tau(V)\leq\tau(W)+\tau(\hat{I})-\tau(W\cap\hat{I})=\tau(\hat{I})=\mu\rho'(V).
\]
All of this shows that $\rho(V)=\rho'(V)$, as desired.
\\[.6ex]
\underline{Case 2:}  $\dim \hat{I}<\dim V-1$.\\ 
Let $x\in V\setminus\hat{I}$. Using that $\rho'(W)\leq\rho'(Z)$ for any subspaces $W,Z$ such that $W\leq Z$, we obtain
\[
  \rho(\hat{I})=\rho'(\hat{I})\leq\rho'(\hat{I}\oplus\subspace{x})\leq\rho'(V)=\rho(\hat{I}),
\]
and hence $\rho(\hat{I})=\rho'(W)$, where $W:=\hat{I}\oplus\subspace{x}$.
Note that~$W$ is $\mu$-dependent thanks to the maximality of~$\hat{I}$.
Furthermore, $\dim\hat{I}=\dim W-1$. Therefore Case~1 yields $\rho'(W)=\rho(W)$.
Now we arrived at $\rho(\hat{I})=\rho(\hat{I}+\subspace{x})$ for all $x\in V$, and \cref{P-RankVx}(a) tells us that 
$\rho(\hat{I})=\rho(V)$. Since $\rho(\hat{I})=\rho'(V)$, this concludes the proof.
\end{proof}

\cref{C-Indep} and \cref{T-rhomax} generalize one direction of \cite[Thm.~8]{JuPe18} where the same properties are proven for the independent spaces of $q$-matroids.
Our next goal is to generalize the other direction of \cite[Thm.~8]{JuPe18}, namely to characterize the collections of spaces along with given rank values that give rise to a \qPM{} having those spaces as independent spaces. 
The following result will be crucial.
It shows that the rank value of any $\mu$-basis of a subspace~$V$ equals the rank value of~$V$.
Recall from \cref{E-MaxRho} that the converse is not true: 
not every $I\in\cI_\mu(V)$ such that $\rho(I)=\rho(V)$ is a $\mu$-basis of~$V$.

\begin{theo}\label{T-Basisrho}
Let~$\cM=(E,\rho)$ be a \qPM{} with denominator~$\mu$.
Let $V\in\cV(E)$. Then 
\[
    \rho(I)=\rho(V)\text{ for all }I\in\cB_\mu(V).
\]
In particular, all $\mu$-bases of~$V$ have the same rank value.
\end{theo}

\begin{proof}
Throughout the proof we will omit the subscript~$\mu$.
The result is clearly true if $V$ is independent. 
Thus, let $V$ be dependent. 
Set $t=\dim V$.
In order to avoid denominators we use again the integer $\rho$-function $\tau:=\mu\rho$.
First of all, there exists
\begin{equation}\label{e-FixJ}
     J\in\cB(V) \text{ such that }\tau(J)=\tau(V).
\end{equation}
Indeed, by \cref{T-rhomax} there exists $J\in\cI(V)$ such that $\tau(J)=\tau(V)$, and by Property~(I2) along with the monotonicity of~$\tau$ 
we may assume that $J\in\cB(V)$. 
Note that by \cref{D-MaxInd} all spaces in $\cB(V)$ have the same dimension, which we denote by~$s$.
\\[.6ex]
\underline{Case 1:}  $s= t-1$. 
Let $I\in\cB(V)$. We want to show that $\tau(I)=\tau(V)$.
Choose a circuit, say~$C$, in~$V$. Then $\tau(C)=\dim C-1$ (see \cref{R-IndJuPe}(c)).
Clearly, $C\not\subseteq I$ by Property~(I2) and thus $C+I=V$ thanks to $\dim I=\dim V-1$.
Furthermore, $C\cap I$ is independent, being a subspace of~$I$,  and thus $\tau(C\cap I)\geq\dim(C\cap I)$.
Using submodularity, we obtain
\begin{align*}
   \tau(V)=\tau(C+I)&\leq \tau(C)+\tau(I)-\tau(C\cap I)\\
             &\leq  \dim C-1+\tau(I)-\dim(C\cap I)\\
             &=\tau(I)+\dim(C+I)-(\dim I+1)\\
             &=\tau(I),
\end{align*}
where the last step follows from $C+I=V$ and $\dim I+1=\dim V$.
All of this implies $\tau(I)=\tau(V)$, and thus all bases of~$V$ have the same rank value.
\\[.6ex]
\underline{Case 2:} $s<t-1$. 
We will show that
\begin{equation}\label{e-tauIJ}
   \tau(I)=\tau(J)\text{ for all }I\in\cB(V),
\end{equation}
where~$J$ is as in~\eqref{e-FixJ}.
We induct on the codimension of $I\cap J$ in~$I$.
Let $\dim(I\cap J)=s-r$, thus $0\leq r\leq s$. 
The case $r=0$ is trivial. 
\\
i) Let $r=1$. Then $I=(I\cap J)\oplus\subspace{x}$ for some $x\in I\setminus J$.
Set $W=J\oplus\subspace{x}$. 
Then $W\leq V$ and $\dim W=\dim J +1$. 
Thus $W$ is dependent by maximality of $J$.
Hence~$I$ and~$J$ are elements of $\cB(W)$, and Case~1 implies $\tau(I)=\tau(J)$.
\\
ii) Assume now $\tau(I)=\tau(J)$ for all $I\in\cB(V)$ such that $\dim(I\cap J)\geq s-(r-1)$ for some $r\geq2$.
Let $I\in\cB(V)$ be such that $\dim(I\cap J)=s-r$.
Choose $K\leq I$ and $x\in I\setminus J$ such that $I=(I\cap J)\oplus K\oplus\subspace{x}$ and
set $I_1=(I\cap J)\oplus K$. 
Then~$I_1$ is independent and $\dim I_1=\dim I-1=\dim J-1$. 
Thanks to Property~(I3) there exists $y\in J\setminus I_1$ such that 
\[
      I':=I_1\oplus\subspace{y}\in\cB(V).
\]
Now we have three bases, $I',\,I,\,J$, of~$V$. 
We show first $\tau(I)=\tau(I')$.
Since $y\not\in I$ we have the subspace
$W:=I\oplus\subspace{y}$ of~$V$, which must be dependent due to maximality of~$I$.
Furthermore, $I,\,I'\leq W$ and $\dim I'=\dim I=\dim W-1$, and therefore
$\tau(I)=\tau(I')$ thanks to Case~1.
Next, we show  $\tau(I')=\tau(J)$.
In order to do so, note that $I'=(I\cap J)\oplus K\oplus\subspace{y}$, where $y\in J$.
Thus $\dim(I'\cap J)\geq s-(r-1)$ and the induction hypothesis yields $\tau(I')=\tau(J)$.
All of this establishes~\eqref{e-tauIJ} and concludes the proof.
\end{proof}

\begin{rem}\label{R-BasesqMatroid}
In a $q$-matroid~$\cM=(E,\rho)$ a subspace $V\in\cV(E)$ satisfies
\[
  \text{$V$is  independent and } \rho(V)=\rho(E)\Longleftrightarrow V\text{ is a basis of }\cM.
\]
The forward direction is in fact the definition of independence in \cite[Def.~2]{JuPe18}.
Thanks to \cref{T-Basisrho} the implication ``$\Longleftarrow$'' holds true for \qPM{}s as well.
However, ``$\Longrightarrow$'' is not true as the \qPM{}s in \cref{E-MaxRho} show.
\end{rem}

We are now ready to provide a  characterization of the pairs $(\cI,\tilde{\rho})$ of collections~$\cI$ of subspaces and rank 
functions~$\tilde{\rho}$ on~$\cI$ that give rise to a \qPM{} whose collection of independent spaces is~$\cI$ and whose rank function restricts to~$\tilde{\rho}$.
Clearly,~$\cI$ has to satisfy (I1)--(I4) from \cref{C-Indep}, and~$\tilde{\rho}$ must satisfy (R1)--(R3).
However, for independence we also need the rank condition from \cref{D-Indep}. 
This leads to (R1$'$) in \cref{T-ExtIndSpaces} below.
Furthermore, since the sum of independent spaces need not be independent, we have to adjust~(R3) and replace 
$\tilde{\rho}(I+J)$ by $\max\{\tilde{\rho}(K)\mid K\in\cI,\,K\leq I+J\}$, thereby accounting for \cref{T-rhomax}.
This results in the submodularity condition (R3$'$) below.
Since one can easily find examples showing that (R1$'$)--(R3$'$) are not sufficient to guarantee submodularity 
of the extended rank function (defined in~\eqref{e-Extrho}), we also have to enforce  \cref{T-Basisrho}.
This leads to condition~(R4$'$), which
states that for any space~$V$ all maximal subspaces that are contained in~$\cI$ 
have the same rank value. 
As we will see, all these conditions together guarantee submodularity of the extended rank function,
and the spaces in~$\cI$ are independent in the resulting \qPM{}.
However, the \qPM{} may have additional independent subspaces; see \cref{E-MoreIndSpaces} below. 
In order to prevent this, we need a natural closure property. 
This will be spelled out in \cref{T-ClosureIndep}.

\begin{theo}\label{T-ExtIndSpaces}
Let $\cI$ be a subset of~$\cV(E)$. For $V\in\cV(E)$ set $\cI(V)=\{I\in\cI\mid I\leq V\}$ and denote by 
$\cIm(V)$ the set of subspaces in~$\cI(V)$ of maximal dimension.
Suppose~$\cI$ satisfies the following.
\begin{mylist}
\item[(I1)\hfill]  $\{0\}\in\cI$.
\item[(I2)\hfill] If $I\in\cI$ and $J\leq I$, then $J\in\cI$.
\item[(I3)\hfill] If $I,\,J\in\cI$ and $\dim I<\dim J$, then there exists $x\in J\setminus I$ such that 
        $I\oplus\subspace{x}\in\cI$.
\item[(I4)\hfill] Let $V,\,W\in\cV(E)$ and $I\in\cIm(V),\,J\in\cIm(W)$.
       Then there exists a space $K\in\cIm(V+W)$ that is contained in $I+J$.
\end{mylist}
Furthermore, let $\tilde{\rho}:\cI\longrightarrow\Q$ and $\mu\in\Q_{>0}$ such that $\mu\tilde{\rho}(I)\in\Z$ for all $I\in\cI$.
Suppose~$\tilde{\rho}$ satisfies the following.
\begin{mylist3}
\item[(R1$'$)\hfill]  $0\leq \mu^{-1}\dim I\leq\tilde{\rho}(I)\leq\dim I$ for all $I\in\cI$.
\item[(R2$'$)\hfill] If $I,J\in\cI$ such that $I\leq J$, then $\tilde{\rho}(I)\leq\tilde{\rho}(J)$.
\item[(R3$'$)\hfill] For all $I,J\in\cI$ we have $\max\{\tilde{\rho}(K)\mid K\in\cI(I+J)\}+\tilde{\rho}(I\cap J)\leq \tilde{\rho}(I)+\tilde{\rho}(J)$.
\item[(R4$'$)\hfill] For all $V\in\cV(E)$ and $I,J\in\cIm(V)$ we have $\tilde{\rho}(I)=\tilde{\rho}(J)$.
\end{mylist3}
Define the map
\begin{equation}\label{e-Extrho}
   \rho:\cV(E)\longrightarrow\Q,\quad V\longmapsto \max\{\tilde{\rho}(I)\mid I\in\cI(V)\}.
\end{equation}
Then $\cM=(E,\rho)$ is a \qPM{} with denominator~$\mu$, and $\cI\subseteq\cI_\mu(\cM)$.
\end{theo}

Note that thanks to~(I3) the set $\cIm(V)$ is the set of maximal elements in the poset $(\cI(V),\leq)$.
Furthermore, by  (R2$'$) and (R4$'$) we have for all $V\in\cV(E)$
\[
      \rho(V)=\tilde{\rho}(I)\ \text{ for any }\ I\in\cIm(V).
\]

\begin{proof}
It is clear that~$\mu$ is a denominator of~$\rho$.
We have to show that~$\rho$ satisfies (R1)--(R3) from \cref{D-PMatroid}.
\\
(R1) Let $V\in\cV(E)$ and $I\in\cI$ such that $I\leq V$ and $\tilde{\rho}(I)=\rho(V)$. Then
$0\leq\tilde{\rho}(I)\leq\dim I\leq\dim V$, which establishes~(R1).
\\
(R2) Let $V,W\in\cV(E)$ be such that $V\leq W$. Let $I\in\cI$ be such that $I\leq V$ and $\tilde{\rho}(I)=\rho(V)$.
Then $I\leq W$ and the definition of~$\rho$ implies $\rho(W)\geq\tilde{\rho}(I)=\rho(V)$, as desired.
\\
(R3) Let $V,W\in\cV(E)$.  Choose $K\in\cIm(V\cap W)$. 
Applying~(I3) repeatedly, we can find $I\in\cIm(V)$ and $J\in\cIm(W)$ such that $K\leq I$ and $K\leq J$.
By~(I4) there exists $H\in\cIm(V+W)$ such that $H\leq I+J$.
Now (R4$'$) implies
\[
  \tilde{\rho}(I)=\rho(V),\quad  \tilde{\rho}(J)=\rho(W),\quad  \tilde{\rho}(H)=\rho(I+J)=\rho(V+W),\quad 
  \tilde{\rho}(K)=\tilde{\rho}(I\cap J)=\rho(V\cap W).
\]
From (R3$'$) we obtain $\rho(I+J)+\tilde{\rho}(I\cap J)\leq \tilde{\rho}(I)+\tilde{\rho}(J)$, and we finally arrive at
\[
  \rho(V+W)+\rho(V\cap W)=  \tilde{\rho}(H)+ \tilde{\rho}(K)= \rho(I+J)+\tilde{\rho}(I\cap J)\leq \tilde{\rho}(I)+\tilde{\rho}(J)
  =\rho(V)+\rho(W),
\]
as desired. Finally, (R1$'$) shows that the spaces in $\cI$ are $\mu$-independent, thus $\cI\subseteq\cI_\mu(\cM)$.
\end{proof}

The following example shows that in general the \qPM{} $\cM$ from  \cref{T-ExtIndSpaces} has more independent spaces 
than~$\cI$.

\begin{exa}\label{E-MoreIndSpaces}
Consider the \qPM{} $\cM=(\F^3,\rho_\cc)$ from \cref{E-IndSpacesRanks}(a).
We have seen already that $\cI_3(\cM)=\cV(\F^3)$.
Define the set  $\cI=\{V\in\cV(\F^3)\mid V\neq \subspace{e_1+e_2,e_3}\text{ and }V\neq\F^3\}$  and let $\tilde{\rho}=\rho_\cc|_{\cI}$.
One easily verifies that $(\cI,\tilde{\rho})$ satisfies (I1)--(I4) and (R1$'$)--(R4$'$).
Furthermore, the extension~$\rho$ defined in \cref{T-ExtIndSpaces} equals $\rho_\cc$ and thus the induced \qPM{} $(\F^3,\rho)$ equals~$\cM$.
Now we have $\cI\subsetneq\cI_3(\cM)$.

\end{exa}

We can easily force equality $\cI=\cI_\mu(\cM)$ by adding the following natural closure property.

\begin{theo}\label{T-ClosureIndep}
Let the pair $(\cI,\tilde{\rho})$ be as in \cref{T-ExtIndSpaces}.
Suppose $(\cI,\tilde{\rho})$ satisfies (I1)--(I4) and (R1$'$)--(R4$'$) as well as 
the following closure property:
\begin{mylist}
\item[(C)\hfill]  
           If $V\in\cV(E)$ is such that
           \begin{romanlist}
           \item all proper subspaces of~$V$ are in~$\cI$,
           \item $\max\{\tilde{\rho}(I)\mid I\in\cI(V)\}\geq\mu^{-1}\dim V$,
           \end{romanlist}           
           then $V$ is in~$\cI$.
\end{mylist}
Then $\cI=\cI_\mu(\cM)$ for the \qPM{} $\cM$ from \cref{T-ExtIndSpaces}. 
\end{theo}

Note that by (I2) and (R1$'$), any subspace $V\in\cI$ satisfies the properties in~(i) and~(ii).

\begin{proof}
Thanks to \cref{T-ExtIndSpaces} it remains to show that any $V\in\cI_\mu(\cM)$ is in~$\cI$.
Recall that $\rho(V)=\max\{\tilde{\rho}(I)\mid I\in\cI(V)\}$.
We induct on $\dim V$.
\\
1) Let $\dim V=1$. Then $\rho(V)\geq\mu^{-1}\dim V$ holds true by the definition of $\mu$-independence, hence~(ii) is satisfied. 
Property~(i) is trivially satisfied by~(I1).
Now~(C) implies $V\in\cI$.
\\
2) Let $\dim V=r$ and assume that all subspaces $V\in\cI_\mu(\cM)$ of dimension at most $r-1$ are in~$\cI$.
Since~$V\in\cI_\mu(\cM)$, the same is true for all its subspaces. 
Hence all proper subspaces are in~$\cI$ by induction hypothesis. 
Again, $\rho(V)\geq\mu^{-1}\dim V$ is true by $\mu$-independence and thus Property~(C) implies that $V\in\cI$.
\end{proof}

\section{Minimal Spanning Spaces and Maximal Independent Spaces}\label{S-SpSp}
By definition, the maximal independent spaces in a \qPM{} are the bases. 
In this section, we introduce spanning spaces and show that -- differently from $q$-matroids -- bases are not the 
same as minimal spanning spaces.
However, spanning spaces turn out to be the dual notion to strongly independent spaces, which we also define in this section.
This result may be regarded as the generalization of the duality result for bases in $q$-matroids.
The latter states that for a $q$-matroid~$\cM$ a space $B$ is a basis of~$\cM$ if and only if $B^\perp$ is a basis of~$\cM^*$.
We show that in fact this equivalence characterizes $q$-matroids within the class of \qPM{}s.

\begin{defi}\label{D-SpSp}
Let $\cM=(E,\rho)$ be a \qPM{}. A subspace $V\in\cV(E)$ is called a \emph{spanning space} if 
$\rho(V)=\rho(E)$.
Furthermore,~$V$ is a \emph{minimal spanning space} if it is a spanning space and no proper subspace is a spanning space.
\end{defi}

For $q$-matroids the minimal spanning spaces are exactly the bases.

\begin{prop}\label{R-SpSp}
Suppose $\cM=(E,\rho)$ is a $q$-matroid. 
A subspace is a basis if and only if it is a minimal spanning space.
As a consequence, all minimal spanning spaces of~$\cM$ have the same dimension.
\end{prop}

\begin{proof}
``$\Rightarrow$'' Let $V$ be a basis of~$\cM$. Then $\dim V=\rho(V)=\rho(E)$. 
For every proper subspace $W\lneq V$ we have $\rho(W)\leq\dim W<\dim V=\rho(V)=\rho(E)$, hence~$W$ 
is not a spanning space. This proves minimality of~$V$.
``$\Leftarrow$''  Let now~$V$ be a minimal spanning space. By definition $\rho(V)=\rho(E)$, and thus it remains to show 
that~$V$ is independent (see \cref{R-BasesqMatroid}).
Suppose~$V$ is dependent.
Then there exists a maximal independent subspace~$W$ of~$V$, and thanks to 
\cref{T-Basisrho} we have $\rho(W)=\rho(V)=\rho(E)$. This contradicts minimality of~$V$.
\end{proof}

The last result is not true for \qPM{}s.

\begin{exa}\label{E-MinSpSp}
\begin{alphalist}
\item The \qPM{} $\cM_\cc(\cC_1)$ from \cref{E-MaxRho}(a) has 126 minimal spanning spaces and they all have dimension~$3$,   
         whereas the bases are the $4$-dimensional spaces.
\item The \qPM{} $\cM$ from \cref{E-MaxRho}(b) has~$4$ (resp.~$10$) minimal spanning spaces of dimension~$1$ (resp.~$2$),
         whereas the bases are the $3$-dimensional subspaces.
\end{alphalist}
\end{exa}

There exist \qPM{}s that are not $q$-matroids and yet  the bases coincide with the minimal spanning spaces
(for instance the \qPM{} $\cM_\cc(\cC)$ in \cite[Ex.~4.2]{GLJ21}).
Thus the equivalence in \cref{R-SpSp} does not characterize $q$-matroids.

The following describes the relation between bases and minimal spanning spaces in a \qPM.

\begin{prop}\label{P-mSpSpBasis}
Let $\cM=(E,\rho)$ be a \qPM{} with denominator~$\mu$.
\begin{alphalist}
\item A minimal spanning space is $\mu$-independent.
\item Every $\mu$-basis of~$\cM$ contains a minimal spanning space and every minimal spanning space is contained in a 
$\mu$-basis.
\end{alphalist}
\end{prop}

\begin{proof}
(a) Let~$V$ be a minimal spanning space. If~$V$ is $\mu$-dependent, then~$V$ contains a maximal $\mu$-independent subspace~$W$, and \cref{T-Basisrho} implies $\rho(W)=\rho(V)=\rho(E)$. This contradicts minimality of~$V$.
\\
(b) is clear.
\end{proof}

Recall the dual \qPM{} from \cref{T-DualqPM}.
Our next result shows that bases are compatible with duality in ``the expected way'' if and 
only if the \qPM{} is a $q$-matroid. 
Part~(a) has been established in \cite{JuPe18}.

\begin{prop}\label{P-BasesDuality}
Let $\cM=(E,\rho)$ be a \qPM{}.
Fix a non-degenerate symmetric bilinear form $\inner{\cdot}{\cdot}$ on~$E$ and let
$\cM^*=(E,\rho^*)$ be the dual of~$\cM$ w.r.t.\ $\inner{\cdot}{\cdot}$. 
\begin{alphalist}
\item If $\cM$ is a $q$-matroid, then for every basis~$B$ of~$\cM$ the orthogonal space $B^\perp$ is a basis of~$\cM^*$.
\item Let~$\mu$ be a denominator of~$\cM$. Suppose there exists a $\mu$-basis $B$ of~$\cM$ such that the orthogonal space $B^\perp$ is a $\mu$-basis of~$\cM^*$. Then $\cM$ is a $q$-matroid.
\end{alphalist} 
\end{prop}

\begin{proof}
(a) has been proven in \cite[Thm.~45]{JuPe18}.
\\
(b)
Let $B$ be a $\mu$-basis of~$\cM$ and $B^\perp$ be a $\mu$-basis of~$\cM^*$. 
Then $\rho(B)=\rho(E)$ and thus $\rho^*(B^\perp)=\dim B^\perp+\rho(B)-\rho(E)=\dim B^\perp$.
\cref{T-Basisrho} implies that every basis $\hat{B}$ of $\cM^*$ satisfies
$\rho^*(\hat{B})=\rho^*(B^\perp)=\dim B^\perp=\dim\hat{B}$.
Now~\eqref{e-rhoVdimV} yields $\rho^*(I)=\dim I$ for all $\mu$-independent spaces~$I$ of~$\cM^*$.
Hence  the dual rank function $\rho^*$ is integer-valued on the $\mu$-independent spaces.
But then the entire rank function~$\rho^*$ is integer-valued thanks to \cref{T-rhomax}.
Now $\rho=\rho^{**}$ is also integer-valued, which means that $\cM$ is a $q$-matroid.
\end{proof}

The above result has an interesting consequence.
Recall from \cref{T-AuxMatroid} the auxiliary $q$-matroid $\cZ_{\cM,\mu}$ of a \qPM{} $\cM$ with denominator~$\mu$.
Part~(b) above implies that if~$\cM$ is a \qPM{} that is not a $q$-matroid, then $\cZ_{\cM^*,\mu}\not\approx\cZ_{\cM,\mu}^*$.
Indeed, from \cref{T-AuxMatroid} we know that a subspace $B\in\cV(E)$ is a $\mu$-basis in~$\cM$ if and only if it is a basis
in $\cZ_{\cM,\mu}$.
Thanks to \cref{P-BasesDuality}(a) the latter is equivalent to $B^\perp$ being in basis in $\cZ_{\cM,\mu}^*$.
But by \cref{P-BasesDuality}(a) $B^\perp$ is not a basis of~$\cM^*$, and thus not of $\cZ_{\cM^*,\mu}$.

As we will show next, Part~(a) can be generalized to \qPM{}s if one replaces bases by minimal spanning spaces 
in~$\cM$ and by maximally strongly independent spaces in~$\cM^*$, where the latter are defined as follows.

\begin{defi}\label{D-StrInd}
Let $\cM=(E,\rho)$ be a \qPM{}. 
A subspace $V\in\cV(E)$ is \emph{strongly independent} if $\rho(V)=\dim V$.
A subspace $V\in\cV(E)$ is \emph{maximally strongly independent} if it is strongly independent and not properly 
contained in a strongly independent subspace.
\end{defi}

From \cref{R-IndepMatroid} we know that strongly independent subspaces are $\mu$-independent for every 
denominator~$\mu$ of~$\cM$.
Furthermore, in $q$-matroids strong independence coincides with independence.
We remark that strongly independent subspaces play a crucial role in \cite{BCIJ21} for the construction of subspace designs.

Now we have the following simple result. 
It shows that spanning spaces and strongly independent spaces are mutually dual.
This may be regarded a
generalization of  \cite[Prop.~83]{BCJ21} and \cite[Thm.~45]{JuPe18}, where the same results have been established for
$q$-matroids.

\begin{prop}\label{P-StrIndepDuality}
Let $\cM$ and~$\cM^*$ be as in \cref{P-BasesDuality} and let $V\in\cV(E)$.
Then $V$ is a (minimal) spanning space in $\cM$ if and only if $V^\perp$ is (maximally) strongly 
independent in $\cM^*$.
\end{prop}

\begin{proof} This follows immediately from $\rho^*(V^\perp)=\dim V^\perp+\rho(V)-\rho(E)$.\end{proof}

In summary, for $q$-matroids the notions `minimal spanning space', `maximally strongly independent space', and `basis' coincide,
whereas these are distinct concepts for \qPM{}s.
We close the paper with a few remarks on the properties -- or rather lack thereof -- of strongly independent spaces and
spanning spaces in \qPM{}s.
In particular, neither the maximally strongly independent spaces nor the minimal spanning spaces behave as well as bases.
This is not surprising since neither collection consists of subspaces of constant dimension
 (see \cref{E-MinSpSp}(b) along with duality).

\begin{rem}
\begin{alphalist}
\item Let $\cM$ be a \qPM{} and $\tilde{\cI}$ be its collection of strongly independent subspaces. 
         The duality in \cref{P-StrIndepDuality} implies that $\tilde{\cI}$ is invariant under taking subspaces, i.e., it satisfies~(I2)
         of \cref{C-Indep} (see also \cref{R-IndepMatroid} and \cite[Lem.~6]{BCIJ21}). 
         It is not hard to find examples showing that $\tilde{\cI}$ does not satisfy~(I3) and~(I4). 
\item  In \cref{C-DepSpacesCircuits} we listed conditions (B1)--(B4) that are satisfied for bases in a \qPM{}. As discussed earlier,
          they give rise to a cryptomorphic definition of $q$-matroids, but not for \qPM{}s  (see \cref{E-MaxRho}).
        For \qPM{}s neither the maximally strongly independent subspaces nor the minimal spanning
        spaces satisfy  (B3) or (B4). 
\end{alphalist}

\end{rem}

\bibliographystyle{abbrv}

\end{document}